\documentclass[12pt, a4paper]{article}
\usepackage[latin1]{inputenc}
\usepackage[T1]{fontenc}
\usepackage[english]{babel}
\usepackage{indentfirst}
\usepackage[dvips]{graphicx}
\usepackage{enumerate}
\usepackage{amstext}
\usepackage{amsfonts}
\usepackage{latexsym}
\usepackage{textcomp}
\usepackage{amssymb}
\usepackage{amsmath}
\usepackage{amscd}
\usepackage[mathcal]{euscript}
\usepackage{geometry}
\usepackage[all]{xy}
\usepackage{anysize}
\usepackage{tabularx}
\usepackage{makeidx}
\usepackage{amsthm}	
\usepackage{mathrsfs}
\usepackage{hyperref}
\usepackage[all]{xy} 
\usepackage{afterpage}
\usepackage{setspace}
\usepackage{epsf}
\usepackage{graphicx}
\usepackage{epsfig}
\usepackage{enumerate}

\newcommand {\e}{\epsilon}

\newcommand {\N}{\mathbb{N}}

\newcommand {\F}{\mathbb{F}}
\newcommand {\G}{\mathrm{G}}

\newtheorem{theorem}{Theorem}[section]
\newtheorem{lema}[theorem]{Lemma}
\newtheorem{corolario}[theorem]{Corollary}
\newtheorem{definition}[theorem]{Definition}
\newtheorem{proposition}[theorem]{Proposition}
\newtheorem{example}[theorem]{Example}
\newtheorem{obs}[theorem]{Remark}

\begin{document}

\title{Free path groupoid grading on Leavitt path algebras}

\maketitle
\begin{center}
{\large Daniel Gonçalves\footnote{This author is partially supported by CNPq.} and Gabriela Yoneda}\\
\end{center}  
\vspace{8mm}

\abstract 
In this work we realize Leavitt path algebras as partial skew groupoid rings. This yields a free path groupoid grading on Leavitt path algebras. Using this grading we characterize free path groupoid graded isomorphisms of Leavitt path algebras that preserves generators. 
\doublespace

\vspace{2.0pc}


\vspace{1.0pc}


\section{Introduction}

Leavitt path algebras, introduced in \cite{arandapino, arandapino1} as generalizations of Cuntz-Krieger algebras, have been the focus of intense research in recent years. Part of the interest in these algebras come from the fact that the Leavitt path algebra associated to a graph encodes much of the combinatorics of the graph and, therefore, their algebraic properties are often linked to combinatorial properties of the underlying graph. To mention a few results in the field, in \cite{Tomforde, TR, vivi} the ideal structure of Leavitt path algebras is studied, in \cite{Chen, GR, uniteq, sepgr, Hranga, ranga} the theory of representations is approached, in \cite{Hazrat} the graded structure of Leavitt path algebras is explored and so on. Of particular interest to our work, in \cite{GR1}, Leavitt path algebras are realized as partial skew group rings (a notion introduced in \cite{Ex,Ex1}) and in \cite{GOR} this realization is used to give new proves of the simplicity criteria of these algebras.

One direction of research in the field regards the study of isomorphisms of Leavitt path algebras and its connections with isomorphisms of their counter-part in C*-algebras, namely graph C*-algebras. This is a topic included in the graph algebra open problem webpage kept by Mark Tomforde and that was initially proposed by Abrams and Tomforde in \cite{AT11}. Recent developments on the subject can be found in \cite{RT}. Also in a very recent development, graph C*-algebras are show to be closely related to orbit equivalence of graphs, see \cite{BCW}. Given the above, the study of isomorphisms of Leavitt path algebras gains extra importance. It is our goal in this paper to give the reader a new insight into Leavitt path algebras and show how this insight can be applied to characterize a class of isomorphisms between Leavitt path algebras. 

A key asset in our work is the theory of partial skew groupoid rings. These rings were defined in \cite{BFP, BP} as generalizations of partial skew group rings, wich in turn were defined in \cite{Ex,Ex1}. As we mentioned before, in \cite{GR1} Leavitt path algebras were realized as partial skew group rings. Building from the ideas in \cite{GR1} we realize Leavitt path algebras as partial skew groupoid rings. While in the group case the group acting is the free group in the edges, in our case the free path groupoid will act. Among the differences between the groupoid approach and the one in \cite{GR1} we mention that, in the partial groupoid skew ring case, each local unit of a Leavitt path algebra is represented on its own fiber and the definition of the appropriate partial action happens in a more natural ideal.

With our description of Leavitt path algebras as partial skew groupoid rings we can obtain a characterization of (free path groupoid) graded isomorphisms of Leavitt path algebras that preserves generators. More precisely, given two graphs $E_1$ and $E_2$ we show that 
if there exists a homomorphism between the associated free path groupoids, say $G_1$ and $G_2$, that preserves the set of finite words and is injective in this set, then there exists a (free path groupoid) graded isomorphism between the Leavitt path algebras (which preserves the generators). Furthermore, in the presence of condition (L) for $E_1$, we show the converse of our statement, that is, we show that if there exists a free path groupoid graded isomorphism, that preserves generators, from $L_K(E_1)$ to $L_K(E_2)$, then there exists an homomorphism between the associated free path groupoids that preserves the set of finite words and is injective in this set. 

We organize our work as follows: In section two we include background material, in order to make the paper as self contained as possible. The characterization of Leavitt path algebras as partial skew groupoid rings is done in section 3 and in section 4 we show the applications mentioned in the previous paragraph.

\section{Background}

\subsection{Partial skew groupoid rings}

The notion of partial skew groupoid rings was introduced in \cite{BFP, BP}, derived from the work by Dokuchaev and Exel in \cite{Ex}. Below we recall the notions leading to partial skew groupoid rings.

\begin{definition} A partial action of a group $\G$ on a set $\Omega$ is a pair $\alpha= (\{D_{t}\}_{t\in \G}, \ \{\alpha_{t}\}_{t\in \G})$, where for each $t\in \G$, $D_{t}$ is a subset of $\Omega$ and $\alpha_{t}:D_{t^{-1}} \rightarrow \Delta_{t}$ is a bijection such that $D_{e} = \Omega$, $\alpha_{e}$ is the identity in $\Omega$, $\alpha_{t}(D_{t^{-1}} \cap D_{s})=D_{t} \cap D_{ts}$ and $\alpha_{t}(\alpha_{s}(x))=\alpha_{ts}(x),$ for all $x \in D_{s^{-1}} \cap D_{s^{-1} t^{-1}}.$ In case $\Omega$ is an algebra or a ring then the subsets $D_t$ should also be ideals and the maps $\alpha_t$ should be isomorphisms. 
\end{definition}

Associated to a partial action of a group $G$ in a ring $A$, the partial skew group ring $A\rtimes_{\alpha} \G$ is defined as the set of all finite formal sums $\sum_{t \in G} a_t\delta_t$, where, for all $t \in G$, $a_t \in D_t$ and $\delta_t$ are symbols. Addition is defined in the usual way and multiplication is determined by $(a_t\delta_t)(b_s\delta_s) = \alpha_t(\alpha_{-t}(a_t)b_s)\delta_{t+s}$.

\begin{definition}[as in \cite{BFP}]
A groupoid is a set $G$ equipped with a binary partially defined operation such that:
\begin{itemize}
\item[(i)] for all $ g, h, l \in G$, $g(hl) \text{ exists if, and only if, } (gh)l$ exists. In this case, $g(hl) = (gh)l$; 
\item[(ii)] for every $ g, h , l \in G$, $g(hl)$ exists if, and only if, $gh \text{ and } hl$ exist;
\item[(iii)] for all $ g \in G$, there exist (unique) elements $d(g), \epsilon(g) \in G$ such that $gd(g) \text{ and } \epsilon(g)g$ exist and $gd(g) = g = \epsilon(g)g$;
\item[(iv)] for each $g \in G$ there exists $g^{-1} \in G$ such that $d(g)= g^{-1}g \text{ and } \epsilon(g) = gg^{-1}$.
\end{itemize}
\end{definition}



\begin{obs}
Let $G$ be a groupoid. The set of admissible pairs is defined as $G^2 = \{ (g,h) \in G \times G : d(g) = \epsilon(h) \}$. An element $e \in G$ is an identity in $G$ if $e = d(g) = \epsilon(g^{-1})$, for some $g \in G$. We denote by $G_0$ the set of all identities in $G$. Notice that $G_0 = \{ g^{-1}g :  g \in G\}$.
\end{obs}

Next, following \cite{BFP,BP} we define partial groupoid actions and their associated partial skew groupoid rings. 

\begin{definition} A partial action $\alpha$ of a groupoid $G$ on a set $X$ is a pair $\alpha =(\{ X_g\}_{g \in G}, \{\alpha_g \}_{g \in G})$ such that:
\begin{itemize}
\item[(I)] for all $g \in G$, $X_{\epsilon(g)}$ is a subset of $X$ and $X_g$ is a subset of $X_{\epsilon(g)}$;
\item[(II)] $\alpha_g : X_{g^{-1}} \rightarrow X_g$ are bijections that satisfies:
\begin{itemize}
\item[(i)] $\alpha_e$ is the identity $\operatorname{Id_{X_e}}$ of $X_e$, for all $e \in G_0$;
\item[(ii)] $\alpha_{h}^{-1}(X_{g^{-1}}\cap X_h)\subseteq X_{(gh)^{-1}}$, whenever $(g,h) \in G^2$;
\item[(iii)] $\alpha_g(\alpha_h(x)) = \alpha_{gh}(x)$, for all $x \in \alpha_h^{-1}(X_{g^{-1}}\cap X_h)$ e $(g,h) \in G^2$.
\end{itemize}
\end{itemize}

For partial actions on a ring $R$ we further ask that each $D_{\epsilon(g)}$ be an ideal of $R$, each $D_g$ be an ideal of $D_{\epsilon(g)}$ and the maps $\alpha_g : D_{g^{-1}} \rightarrow D_g$ be isomorphisms.
\end{definition}


\begin{definition} Let $\alpha\ = (\{D_g\}_{g \in G}, \{ \alpha_g\}_{g \in G})$ be a partial action of the groupoid $G$ on a ring $R$. The partial skew groupoid ring $R \rtimes_{\alpha}G$ is the set of all formal sums of the form $\displaystyle\sum_{g \in G} a_g \delta_g$, where $a_g \in D_g$, with addition defined in the usual way and multiplication given by
\begin{equation*}
a_g\delta_g \cdot b_h \delta_h =
\begin{cases}
 \alpha_g(\alpha_{g^{-1}}(a_g)b_h)\delta_{gh}& \text{, if } (g,h) \in G^2 ;\\
0 & \text{, otherwise. } 
\end{cases}
\end{equation*}
\end{definition}


\subsection{Leavitt path algebras as partial skew group rings}

A directed graph $E=(E^0, E^1, r, s)$ consists of a set $E^0$ of \emph{vertices}, a set $E^1$ of \emph{edges}, a \emph{range map} $r : E^1 \to E^0$ and a \emph{source map} $s: E^1 \to E^0$ which may be used to read off the direction of an edge.
Given a field $K$ and a directed graph $E$, the so called \emph{Leavitt path algebra} associated with $E$ (see e.g. \cite{arandapino,arandapino1})
is denoted by $L_K(E)$.
To be more precise,
$L_K(E)$ is the universal $K$-algebra generated by a set $\{v, e, e^* : v \in E^0, e\in E^1\}$
of elements satisfying the following five assertions:
\begin{enumerate}[{\rm (i)}]
	\item for all $v,w\in E^0$, $v^2=v$, and $vw=0$ if $v\neq w$;
	\item $s(e)e=er(e)=e$ for all $e \in E^1$;
	\item $r(e)e^*=e^*s(e)=e^*$ for all $e \in E^1$;
	\item for all $e,f \in E^1$, $e^*e=r(e)$, and $e^* f = 0$ if $e\neq f$;
	\item $v=\sum_{\{e\in E^1 \, : \,s(e)=v\}} ee^*$ for each vertex $v\in E^0$ which satisfies
	$0 < \#\{e \in E^1 : s(e) = v\} < \infty$.
\end{enumerate}

In \cite{GR1}, it was showed that each
Leavitt path algebra can be realized as a partial skew group ring.
We shall review this construction by first defining a partial action at the level of sets.

Let $E=(E^0,E^1,r,s)$ be a directed graph.
A \emph{path of length $n$} in $E$ is a sequence $\xi_1 \xi_2 \ldots \xi_n$ of edges in $E$ such that $r(\xi_i)=s(\xi_{i+1})$ for $i \in \{1,2,\ldots,n-1\}$.
If $\xi$ is a path of length $n$, then we write $|\xi|=n$.
The set of all finite paths in $E$ is denoted by $W$.
An \emph{infinite path} in $E$ is an infinite sequence $\xi_1 \xi_2 \ldots$ of edges in $E$ such that $r(\xi_i)=s(\xi_{i+1})$ for $i \in \N$.
The set of all infinite paths in $E$ is denoted by $W^\infty$.
Note that $W$ (respectively $W^\infty$) is a subset of the set of all finite (respectively infinite) words in the alphabet $E^1$.
As usual, the range and source maps can be extended from $E^1$ to $W \cup W^\infty \cup E^0$
by defining $s(\xi):=s(\xi_1)$ for $\xi=\xi_1\xi_2\ldots \in W^\infty$ or $\xi=\xi_1\ldots\xi_n \in W$,
$r(\xi):=r(\xi_n)$ for $\xi=\xi_1\ldots\xi_n \in W$ and $r(v)=s(v)=v$ for $v\in E^0$.
A finite path $\eta$ is said to be an \emph{initial subpath} of a (possibly infinite) path $\xi$, if there is a path $\xi'$ such that $r(\eta)=s(\xi')$ and $\xi=\eta\xi'$ hold.

The partial action that we are about to define, takes place on the set
\begin{displaymath}
X=\{\xi\in W:r(\xi) \text{ is a sink }\}\cup \{v\in E^0: v\text{ is a sink }\}\cup W^{\infty}	
\end{displaymath}
which is acted upon by $\F$, the free group generated by the set $E^1$.
(notice that, since $\F$ is generated by $E^1$, some elements of $\F$ can be thought of
as coming directly from $W$).

In order to have a partial action of $\F$ on $X$, for each $c\in \F$ we need to define a set $X_c$ and a map $\theta_c : X_{c^{-1}} \to X_c$,
such that they comply with the definition of a partial action. This is done as follows:

\begin{itemize}
\item $X_0:=X$, where $0$ is the neutral element of $\F$. 

\item $X_{b^{-1}}:=\{\xi\in X: s(\xi)=r(b)\},$ for all $b\in W$. 

\item $X_a:=\{\xi\in X: \xi_1\xi_2...\xi_{|a|}=a\},$ for all $a\in W$.

\item $X_{ab^{-1}}:=\{\xi\in X: \xi_1\xi_2...\xi_{|a|}=a\}=X_a,$ for $ab^{-1}\in \F$ with $a,b\in W$, $r(a)=r(b)$ and $ab^{-1}$ in its reduced form.

\item $X_c:=\emptyset$, for all other $c \in \F$.

\end{itemize}

The second step towards the construction of our partial action, is to define the maps $\theta_c : X_{c^{-1}} \to X_{c}$, for $c\in \F$.

Let  $\theta_{0}:X_{0}\rightarrow X_{0}$ be the identity map. For $b\in W$,  $\theta_b:X_{b^{-1}}\rightarrow X_b$ is defined by $\theta_b(\xi)=b\xi$.
and $\theta_{b^{-1}}:X_b\rightarrow X_{b^{-1}}$, for $\eta \in X_{b^{-1}}$, is defined by 
$\theta_{b^{-1}}(\eta)= \eta_{|b|+1}\eta_{|b|+2}...$, if $r(b)$ is not a sink and $\theta_{b^{-1}}(b)=r(b)$, if $r(b)$ is a sink.

Finally, for $a,b\in W$ with $r(a)=r(b)$ and $ab^{-1}$ in reduced form, $\theta_{ab^{-1}}:X_{ba^{-1}}\rightarrow X_{ab^{-1}}$ is defined by $\theta_{ab^{-1}}(\xi)=a\xi_{(|b|+1)}\xi_{(|b|+2)}\ldots$ for $\xi \in X_{ba^{-1}}$.
It is not difficult to see that the inverse of this map is given by 
$\theta_{ba^{-1}}:X_{ab^{-1}}\rightarrow X_{ba^{-1}}$ which is defined by $\theta_{ba^{-1}} (\eta)=b\eta_{(|a|+1)}\eta_{(|a|+2)}...$ for $\eta \in X_{ab^{-1}}$.

Notice that $\{\{X_c\}_{c\in \F}, \{\theta_c\}_{c\in\F}\}$ is a partial action on the level of sets and so it induces a partial action $\{\{F(X_c)\}_{c\in \F}, \{\alpha_c\}_{c\in\F}\}$, where, for each $c\in \F$, $F(X_c)$ denotes the algebra of all functions from $X_c$ to $K$, and $\alpha_c:F(X_{c^{-1}})\rightarrow F(X_c)$ is defined by $\alpha_c(f)=f\circ \theta_{c^{-1}}$. The partial skew group ring associated to this partial action is not $L_K(E)$ yet. For this one proceeds in the following way:

For each $c\in \F$, and for each $v\in E^0$, define the characteristic maps $1_c:=\chi_{X_c}$ and $1_v:=\chi_{X_v}$, where $X_v=\{\xi\in X:s(\xi)=v\}$. Notice that $1_c$ is the unit of $F(X_c)$. Finally, let
$$D_0=\text{span}\{\{1_p:p\in \F\setminus\{0\}\}\cup\{1_v:v\in E^0\}\},$$  (where $\text{span}$ means the $K$-linear span) and, for each $p\in \F\setminus\{0\}$, let $D_p\subseteq F(X_p)$ be defined as $1_p D_0$, that is, $$D_p=\text{span}\{1_p1_q:q\in \F\}.$$
Since $\alpha_p(1_{p^{-1}}1_q)=1_p1_{pq}$ (see \cite{GR1}), consider, for each $p\in \F$, the restriction of $\alpha_p$ to $D_{p^{-1}}$. Notice that $\alpha_{p}:D_{p^{-1}}\rightarrow D_p$ is an isomorphism of $K$-algebras and, furthermore, $\{\{\alpha_p\}_{p\in \F}, \{D_p\}_{p\in \F}\}$ is a partial action.

In \cite{GR1} it was shown that the partial skew group ring $D_0\rtimes_\alpha\F$ is isomorphic to the Leavitt path algebra $L_K(E)$.
More precisely, the map $\varphi : L_K(E) \to D_0 \rtimes_\alpha \F$ defined by $\varphi(e)=1_e\delta_e$, $\varphi(e^*)=1_{e^{-1}}\delta_{e^{-1}}$ for all $e\in E^1$ and $\varphi(v)=1_v\delta_0$ for all $v\in E^0$, was shown to be a $K$-algebra isomorphism.


\section{Leavitt Path algebras as partial skew groupoid rings}

Associated to a graph one can construct the free path groupoid. In this section we use this groupoid to realize the Leavitt path algebra associated to a graph as a partial skew ring groupoid ring. It is interesting to note that, under this realization, there exists a standard set of local units of the Leavitt path algebra such that each local unit lives on its own fiber, while if the Leavitt path algebra is realized as a partial skew group ring all local units live in the fiber associated to the neutral element of the free group. 

The definition of the free path groupoid can be found in \cite{PH}. Since this is the main groupoid in our work we recall its construction below.


Let $E = (E^0, E^1, r, s)$ be a graph and recall that $W$ denotes the set of finite paths in $E$. Let $(E^1)^*:=\{e^*:e\in E^1\}$ and extend the maps $r,s$ to $(E^1)^*$ by $s(e^*) = r(e) \text{ e } r(e^*) = s(e)$, for all $e \in E^1$.
For $\alpha = \alpha_1 \ldots \alpha_n \in W$ define $\alpha^* := \alpha_n^* \ldots \alpha_1^*$.
Let $$P = \{ \alpha_1 \ldots \alpha_n : \alpha_i  \in E^1 \cup (E^1)^* \cup E^0 \text{ e } r(\alpha_i) = s(\alpha_{i+1})\}.$$ 

\begin{example} In the graph below the element $e_1e_2e_3^{*}e_4$ belongs to $P$.
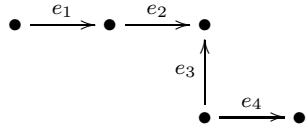
$$ \xymatrix{ \bullet \ar[r]^{e_1}  & \bullet \ar[r]^{e_2} & \bullet \\ &  & \bullet \ar[u]^{e_3} \ar[r]^{e_4} & \bullet } $$
\end{example}

In $P$ we define a concatenation operation, denoted by $\cdot $, so that 
\begin{itemize}
\item[] $e_1 \cdot e_2 = e_1e_2$ (concatenation);
\item[] $e \cdot r(e) = s(e) \cdot e = e$;
\item[] $e \cdot e^* = s(e)$;
\item[] $e^* \cdot e = r(e)$;
\end{itemize}
and we extend this operation to finite paths in the expected way. 

\begin{definition}
Given $a_1, \ldots, a_n , b \in E^1 \cup (E^1)^* \cup E^0$ and sequences of the form
\[ a_1 \cdots a_k b b^* a_{k+1}a_n \]
or
\[ a_1 \cdots a_k r(a_k) a_{k+1}\cdots a_n = a_1 \cdots a_k s(a_{k+1})a_{k+1} \cdots a_n \]
in $P$, we say that $a_1 \cdots a_k a_{k+1}a_n$ is a reduction of the these sequences. 
A sequence in $P$ is said irreducible if it can not be reduced. 
\end{definition}

\begin{definition}\label{exemplogrupoide} Let $G$ be the set of all irreducible sequences in $P$, $G^2 = \{(\alpha,\beta) \in G \times G : r(\alpha) = s(\beta) \}$ and define the partial groupoid operation  $G^2 \rightarrow G$ by $\alpha \cdot \beta = \operatorname{irr}(\alpha\beta)$, where $\operatorname{irr}(\alpha\beta)$ denotes the reduction of $\alpha\beta$. Then $G$ is the free path groupoid associated to the graph $E$.
\end{definition}

\begin{obs} To emphasize the notations we are using notice that, for all $\alpha \in G$, we have $s(\alpha) = \e(\alpha)$, $r(\alpha) = d(\alpha)$, $\alpha \alpha^* = s(\alpha) \text{ and } \alpha^* \alpha = r(\alpha).$
\end{obs}

%







We now proceed to realize a Leavitt path algebra as a partial skew groupoid ring. The main ideas follow what was done for groups. We start by defining a partial groupoid action in the level of sets. 

Let $E = (E^0, E^1, r, s)$ be a graph and $G$ the associated free path groupoid, as in \ref{exemplogrupoide}. Let $X = \{ \xi \in W : r(\xi) \text{ is a sink}\} \cup \{v \in E^0 : v \text{ is a sink}\} \cup W^{\infty}$. For $a \in W \setminus E^0$, define $X_a := \{ \xi \in X: \xi_1 \ldots \xi_{|a|} = a\}$ and $X_{a^{-1}} := \{ \xi \in X: r(\xi) = r(a)\}$. For $ab^{-1} \in G$, with $a, b \in W \text{ and } a, b \notin E^0$, define $X_{ab^{-1}}:= X_a$. For $v \in E^0$, define $X_v := \{ \xi \in X: s(\xi) = v\}$. Finally, for any other $g \in G$, let $X_g = \varnothing$.

\begin{obs} Notice that the sets $X_v$ were not part of the definition of the partial action of the free group. So, to obtain the Leavitt path algebra, it was necessary to artificially add the span of the characteristic functions of the sets $X_v$ to the algebra where the free group acts (see the definition of $D_0$ in the previous section). This procedure will be avoided with our groupoid approach.
\end{obs}

Next we define the groupoid partial action.

For $v \in E^0$, define
\begin{center}\begin{tabular}{rrrl} 
$\theta_{v}:$ & $X_{v}$ & $\rightarrow$ & $X_v$ \\
&  $\xi$ & $\mapsto$ & $\xi.$
\end{tabular}
\end{center}
 
For $b \in W$, define
\begin{center}\begin{tabular}{rrrl} 
$\theta_{b}:$ & $X_{b^{-1}}$ & $\rightarrow$ & $X_b$ \\
&  $\xi$ & $\mapsto$ & $b\xi,$
\end{tabular}
\end{center}

and $\theta_{b^{-1}} : X_b \rightarrow X_{b^{-1}}$,  by 
\begin{equation*}
\theta_{b^{-1}}(\xi) =
\begin{cases}
\xi_{|b|+1}\xi_{|b|+2} \ldots & \text{, if } r(b) \text{ is not a sink}\\
r(b) & \text{, if } r(b) \text{ is a sink}.
\end{cases}
\end{equation*}

For $ab^{-1} \in G$, define
\begin{center}\begin{tabular}{rrrl}
$\theta_{ba^{-1}}:$ & $X_{ab^{-1}}$ & $\rightarrow$ & $X_{ba^{-1}}$ \\
 & $\xi$ & $\mapsto$ & $b\xi_{|a|+1}\xi_{|a|+2} \ldots $
\end{tabular}
\end{center}
\begin{center}\begin{tabular}{rrrl}
$\theta_{ab^{-1}}:$ & $X_{ba^{-1}}$ & $\rightarrow$ & $X_{ab^{-1}}$ \\
 & $\xi$ & $\mapsto$ & $a\xi_{|b|+1}\xi_{|b|+2} \ldots $.
\end{tabular}
\end{center}

\begin{proposition} $\theta = (\{ X_g\}_{g \in G} , \{ \theta_g\}_{g \in G})$ is a partial action of the groupoid $G$ in the set $X$.
\end{proposition}
\begin{proof}

It is straightforward to check that $X_{\e(g)}$ is a subset of $X$, that $X_g$ is a subset of $X_{\e(g)}$ and that $\theta_g : X_{g^{-1}} \rightarrow X_g$ is a bijection. We need to prove the three conditions a partial groupoid action must satisfy.


\begin{itemize}
\item[(i)] For all $v \in E^0$ we have that $\theta_v$ is the identity in $X_v$. 

Given $\xi \in X_v$, then $s(\xi) = v = r(\xi)$ and hence
\[ \theta_v(\xi) = v\xi = s(\xi)\xi = \xi . \]

\item[(ii)] $\theta_{h}^{-1}(X_{g^{-1}} \cap X_h) \subseteq X_{(gh)^{-1}}$.

Let $g = ab^{-1} , h = cd^{-1} \in G$. 

Since $X_{g^{-1}} \cap X_h = X_{ba^{-1}} \cap X_{cd^{-1}}$, we need to consider three cases: 

\textbf{Case 1:} If $b$ is not the beginning of $c$ and $c$ is not the beginning of $b$.

In this case, $X_{ba^{-1}} \cap X_{cd^{-1}} = \varnothing$.

\textbf{Caso 2:} If $b$ is the beginning of $c$ (we write $c = bc'$).

Notice that in this case we have that 

\begin{equation*}
X_{ba^{-1}} \cap X_{cd^{-1}} =
\begin{cases}
X_{a^{-1}} \cap X_{d^{-1}} & \text{, if } b =r(a) , c = r(d) \in E^0   \\
X_{a^{-1}} \cap X_c & \text{, if }  b \in E^0 , c \notin E^0 \\
X_b \cap X_c & \text{, if } b, c \notin E^0
\end{cases}
\end{equation*}
\begin{equation*}
=
\begin{cases}
X_{a^{-1}} \cap X_{d^{-1}} & \text{, if } b =r(a) , c = r(d) \in E^0   \\
X_{a^{-1}} \cap X_c & \text{, if }  b \in E^0 , c \notin E^0 \\
X_c & \text{, if } b, c \notin E^0
\end{cases}
\end{equation*}

Then,
\begin{equation*}
\theta_{h}^{-1}(X_{g^{-1}} \cap X_h) = 
\begin{cases}
\theta_{d^{-1}}^{-1}(X_{a^{-1}} \cap X_{d^{-1}}) \hspace{0.3cm} (1)\\
\theta_{cd^{-1}}( X_{a^{-1}} \cap X_c) \hspace{0.3cm} (2) \\
\theta_{cd^{-1}}^{-1}(X_c) \hspace{0.3cm} (3)
\end{cases}
\end{equation*}

In $(1)$ we have that
$ (gh)^{-1} = h^{-1}g^{-1} = dc^{-1}ba^{-1} = da^{-1}$
and hence $X_{(gh)^{-1}} = X_{da^{-1}} = X_d$. In $(2)$ we have that $(gh)^{-1} = h^{-1}g^{-1} = dc^{-1}ba^{-1} = dc^{-1}a^{-1}$ and hence $X_{(gh)^{-1}} = X_{dc^{-1}a^{-1}} = X_d$. Finally, in $(3)$ we have that $(gh)^{-1} = dc^{-1}ba^{-1} = dc'^{-1}b^{-1}ba^{-1} = dc'^{-1}r(b)a^{-1} = dc'^{-1}a^{-1}$. In any of the tree cases we obtain that $$\theta_{dc^{-1}} \subseteq X_{dc^{-1}} = X_{dc'^{-1}b^{-1}} = X_{(gh)^{-1}},$$ as desired.

\textbf{Case 3:} $c$ is the beginning of $b$ (we write $b = cb'$).

This case is analogous to case 2.

\item[(iii)] For all $\xi \in \theta_g^{-1}(X_g \cap X_{h^{-1}})$ we have that $\theta_h \circ \theta_g (\xi) = \theta_{hg}(\xi)$. 

This can be proved by separating into three cases as before and then verifying the desired equality. 








\end{itemize}

\end{proof}

We can now define the partial action in the ring level: 

Let $K$ be a field and let $\mathcal{F}(X)$ denote the set of all functions from $X$ to $K$. For each $X_g \neq 0$, let $\mathcal{F}(X_g) = \{ f \in \mathcal{F}(X): f \text{ vanishes outside of } X_g \} \subseteq \mathcal{F}(X)$ and for $X_g = \varnothing$, let $\mathcal{F}(X_g) = \{ \text{ null function }\}$. Now define 
\begin{center}\begin{tabular}{rrrl} 
$\alpha_g:$ & $\mathcal{F}(X_{g^{-1}})$ & $\rightarrow$ & $\mathcal{F}(X_g)$ \\
&  $f$ & $\mapsto$ & $f \circ \theta_{g^{-1}}$
\end{tabular}
\end{center}
and
\begin{center}\begin{tabular}{rrrl} 
$\alpha_v:$ & $\mathcal{F}(X_v)$ & $\rightarrow$ & $\mathcal{F}(X_v)$ \\
&  $f$ & $\mapsto$ & $f \circ \theta_v$
\end{tabular}
\end{center}

\begin{proposition}
$\alpha = (\{\mathcal{F}(X_g)\}_{g \in G}, \{ \alpha_g\}_{g \in G})$ is a partial action of the groupoid $G$ on the ring $\mathcal{F}(X)$.
\end{proposition}
\begin{proof}

The proof of this proposition follows the usual techniques used in the partial skew group ring case, see for example \cite{ELQ, vivi} or \cite{Gonc}.








\end{proof}

As it happened in the group case, the partial action above is too "large". To get the Leavitt path algebra we need to do the following: Let $D(X) = \operatorname{span} \{ 1_g : g \in G\}$, $D_p = 1_pD(X) = \operatorname{span}\{ 1_p1_g : g \in G\}$, for all $p \in G$, and consider the restriction of $\alpha_p$ to the ideals $D_p$, 
\begin{center}\begin{tabular}{rrrl} 
$\alpha_p:$ & $D_{p^{-1}}$ & $\rightarrow$ & $D_p$ \\
&  $1_{p^{-1}}1_q$ & $\mapsto$ & $\alpha_p(1_{p^{-1}}1_p) = 1_p1_{pq}$
\end{tabular}
\end{center}

\begin{definition} Let $\tilde\alpha = (\{D_g\}_{g \in G} , \{ \alpha_g\}_{g \in G})$ be the partial action of the groupoid $G$ on the ring $D(X)$, as described above. Then we will denote the partial skew groupoid ring associated to $\tilde\alpha$ by $D(X)\rtimes_{\tilde{\alpha}}G $.
\end{definition}

\begin{obs} As we remarked before, when dealing with the partial action of the free group $\mathbb{F}$, it was necessary to define $D(X) = \operatorname{span}\{  \{ 1_p : p \in \mathbb{F} \setminus \{0\}\} \cup \{ 1_v : v \in E^0 \}\}$. With the groupoid approach the definition of $D(X)$ is more natural. Notice also that for any vertex $v$ the function $1_v \delta_v \in D(X) \rtimes_{\tilde{\alpha}} G$.
\end{obs}

\begin{theorem}\label{propnova}
Let $E$ be a graph. Then $D(X) \rtimes_{\tilde{\alpha}} G$ is isomorphic to $L_K(E)$ via an isomorphism $\phi$ such that, for all $e \in E^1$, $\phi(e) = 1_e\delta_e$, $\phi(e^*) = 1_{e^{-1}}\delta_{e^{-1}}$ and, for all $v \in E^0$, $\phi(v)=1_v\delta_v$.
\end{theorem}
\begin{proof}

Consider the family $\{1_e\delta_e  ,  1_{e^{-1}}\delta_{e^{-1}}, 1_v\delta_v  : e \in E^1, v \in E^0 \}$ in $D(X) \rtimes_{\tilde{\alpha}}G$. This family satisfy the relations defining the Leavitt path algebra $L_K(E)$. To give the reader an idea of the techniques involved we show how to prove the Cuntz-Krieger relation.

Let $v \in E^0$ be such that $0 < \# \{ e: s(e) = v\} < \infty$. Since $X_v = \displaystyle\bigcup_{e \in E^1  : s(e)= v}X_e$, we have that 
\[ \hspace{-25pt}\displaystyle\sum_{e \in E^1: s(e) = v} 1_e\delta_e1_{e^{-1}}\delta_{e^{-1}} =\sum_{e \in E^1: s(e) = v} 1_e\delta_{s(e)} = (\sum_{e \in E^1: s(e) = v}1_e )\delta_v = 1_v\delta_v . \]

Hence, by the universal property of $L_K(E)$ and the fact that $L_K(E) \cong D_0 \rtimes_{\alpha} \mathbb{F}$ via an isomorphism such that, for all $e \in E^1$, $\phi(e) = 1_e\delta_e$, $\phi(e^*) = 1_{e^{-1}}\delta_{e^{-1}}$ and, for all $v \in E^0$, $\phi(v)=1_v\delta_v$ (see section 2), we obtain an unique homomorphism \[\Gamma: D_0 \rtimes \mathbb{F} \to D(X) \rtimes G\] such  that $\Gamma(1_e\delta_e) = 1_e\delta_e$, $\Gamma(1_{e^{-1}}\delta_{e^{-1}}) = 1_{e^{-1}}\delta_{e^{-1}}$, for all $e \in E^1$, and $\Gamma(1_v\delta_0) = 1_v \delta_v$, for all $v \in E^0$.



It is clear that $\Gamma$ is surjective, since $\{ 1_e\delta_e, 1_{e^{-1}}\delta_{e^{-1}} : e \in E^1\}$ e $\{1_v\delta_v : v \in E^0 \}$ generates $D(X)\rtimes_{\alpha} G$. We will show that $\Gamma$ is injective.

Let $x \in \operatorname{Ker}\Gamma$. Then we can write $x = a_0\delta_0 + \displaystyle\sum a_g \delta_g$, where $a_g \in D_g$ e $a_0 = \left(\displaystyle\sum \lambda_{ab^{-1}}1_{ab^{-1}} + \displaystyle\sum \beta_w1_w\right) \in D_0$.

For $0 \neq g$, since $\Gamma$ is a homomorphism, it is straightforward to check that $\Gamma(a_g\delta_g) = a_g\delta_g$. We need to compute $\Gamma(a_0 \delta_0)$. 

Notice that given $1_a\delta_0 \in D_0\delta_0$, with $a \in W$, we have that 
\begin{align*}
1_a\delta_a \cdot 1_{a^{-1}}\delta_{a^{-1}} &= \alpha_a(\alpha_{a^{-1}}(1_a)1_{a^{-1}})\delta_0 \\
&= \alpha_a(1_{a^{-1}}1_{a^{-1}})\delta_0 \\
&= 1_a\delta_0
\end{align*}
and hence
 \begin{align*}
 \Gamma(1_a\delta_0) &= \Gamma(1_a\delta_a)\Gamma(1_{a^{-1}}\delta_{a^{-1}}) \\
 &= 1_a\delta_a \cdot 1_{a^{-1}}\delta_{a^{-1}} \\
 &= \alpha_a(\alpha_{a^{-1}}(1_a)1_{a^{-1}})\delta_{aa^{-1}} \\
 &= \alpha_a(1_{a^{-1}})\delta_{s(a)} \\
 &= 1_{a}\delta_{s(a)}.
 \end{align*}
Therefore, since $1_{cd^{-1}} = 1_c \text{ and } 1_{d^{-1}} = 1_{r(d)}$ for $c, d \in W$, we have that
\begin{align*}
\Gamma(a_0\delta_0) &= \Gamma((\displaystyle\sum \lambda_{ab^{-1}}1_{ab^{-1}} + \displaystyle\sum \beta_w1_w)\delta_0) \\
&= \displaystyle\sum \lambda_{ab^{-1}}1_{ab^{-1}}\delta_{s(a)} + \displaystyle\sum \beta_w1_w\delta_w \\
&= \displaystyle\sum_{v \in V'} 1_va_0\delta_v ,
\end{align*}
where $V' = \{  v \in V^0 : v = s(ab^{-1}) \text{ for } \lambda_{ab^{-1}} \text{ or } \beta_v \neq 0\}$.
 
We conclude that
\begin{align*}
\Gamma(x) = 0 &\iff \displaystyle\sum_{v \in V'} 1_va_0\delta_v + \displaystyle\sum a_g\delta_g = 0 \\
&\iff a_0 = 0 \text{ and } a_g = 0, \text{ for all } g 
\end{align*}
and hence $\Gamma $ is injective as desired.
\end{proof}

\section{Applications to free path groupoid graded isomorphisms of Leavitt path algebras.}

Given two graphs, say $E_1$ and $E_2$, by the results of the previous section, we can see the associated Leavitt path algebras as groupoid graded rings. In this section we show that if there exists a homomorphism between the associated free path groupoids, say $G_1$ and $G_2$, that preserves the set of finite words and is injective in this set, then there exists a (free path) groupoid graded isomorphism between the Leavitt path algebras (which preserves the generators). Furthermore, in the presence of condition (L) for $E_1$, we show the converse of our statement, that is, we show that if there exists a free path groupoid graded isomorphism, that preserves generators, from $L_K(E_1)$ to $L_K(E_2)$, then there exists an homomorphism between the associated free path groupoids that preserves the set of finite words and is injective in this set.   

We start the work by giving a few necessary definitions and showing a series of auxiliary results. 

\begin{definition}
Let $G, H$ be groupoids. A map $h: G \to H$ is a groupoid homomorphism if $(g_1, g_2) \in G^2$ implies that $(h(g_1), h(g_2)) \in H^2$ and $h(g_1g_2) = h(g_1)h(g_2)$.
\end{definition}

\begin{definition}\label{Sset} Given a graph $E$, let $G$ be the free path groupoid associated and $\theta= (\{ X_g\}_{g \in G} , \{ \theta_g\}_{g \in G})$ the partial groupoid action of $G$ in the set $X$, as defined in section 3. We define $S\subseteq G$ by \[S = \{g \in G : X_g \neq \varnothing\}.\]
\end{definition}

Throughout this section, for $i=1,2$, $E_i = (E^0_i, E^1_i, r, s)$ is a graph, $G_i$ is the free path groupoid associated to $E_i$ (as defined in \ref{exemplogrupoide}), $W_i \subseteq G_i$ is the set of all finite paths in $E_i$, $\theta_i = (\{ X_g\}_{g \in G_i} , \{ \theta_g\}_{g \in G_i})$ is the partial groupoid action of $G_i$ in the set $X_i$, as defined in section 3, and $S_i$ is associated with $G_i$ in the fashion of definition \ref{Sset} . Furthermore, $h: G_1 \to G_2$ is a groupoid homomorphism.

\begin{lema}\label{lemarangesource}
Let $G_i$ and $W_i$ be as above and $h: G_1 \to G_2$ be a groupoid homomorphism. Then for all $v \in E_1^0$, we have that $h(v) \in E_2^0$. Furthermore, for all $\alpha \in W_1$ we have that $h(s(\alpha)) = s(h(\alpha))$ and $h(r(\alpha)) = r(h(\alpha))$.
\begin{proof}

Let $\alpha \in G_2$ be such that $h(v) = \alpha$, where $v \in E_1^0$. Then 
\[\alpha = h(v) = h(vv) = h(v)h(v) = \alpha \alpha \]
and therefore $\alpha \in E_2^0$.

Now, since $h(\alpha) = h(s(\alpha)\alpha) = h(s(\alpha))h(\alpha)$, we have that
$s(h(\alpha))= h(s(\alpha))$. The equality involving $r(\alpha)$ follows analogously. 
\end{proof}
\end{lema}

\begin{corolario}
Given $\alpha \in W_1$, it follows that $h(\alpha^{-1}) = h(\alpha)^{-1}$ and hence, for every $(\alpha, \beta^{-1}) \in G_1^2$, it holds that $h(\alpha\beta^{-1}) = h(\alpha)h(\beta)^{-1}$.
\end{corolario}


Assuming that the homomorphism $h$ is injective we can also prove the following:

\begin{lema} Suppose that $h: G_1 \to G_2$ is injective in $W_1$. Then $(\alpha, \beta) \in G_1^2$ if, and only if, $(h(\alpha),h(\beta)) \in G_2^2$. Furthermore, for every $\alpha, \beta \in W_1$, we have that $r(\alpha) = r(\beta)$ if, and only if, $r(h(\alpha)) = r(h(\beta))$.\label{lemacomponivel}
\end{lema}

\begin{proof}
Suppose that $(\alpha, \beta) \notin G_1^2$. Then $r(\alpha) \neq s(\beta)$ and, since $h$ is injective in $W_1$, it follows that 
\[r(h(\alpha)) = h(r(\alpha)) \neq h(s(\beta)) = s(h(\beta))\]
and hence $(h(\alpha), h(\beta)) \notin G_2^2$.

Now suppose that $r(\alpha) = r(\beta)$. Then $(\alpha, \beta^{-1}) \in G_1^2$ and, since $h$ is a homomorphism, we have that $(h(\alpha), h(\beta^{-1})) \in G_2^2$. Therefore $r(h(\alpha)) = s(h(\beta^{-1})) = r(h(\beta^{-1})^{-1})) = r(h(\beta))$.

Finally, suppose that $r(\alpha) \neq r(\beta)$. Then, since $h$ is injective in $W_1$, we have that $h(r(\alpha)) \neq h(r(\beta))$ and hence
\[r(h(\alpha)) = h(r(\alpha)) \neq h(r(\beta)) = r(h(\beta)).\qedhere\]

\end{proof}


If we further assume that the homomorphism $h$ preserves the set of finite paths in the graph $E_1$ we can prove the following:

\begin{lema}\label{lematamanho}
Suppose that $h: G_1 \to G_2$ is injective in $W_1$ and $h(W_1) = W_2$. Then, for every $\alpha \in W_1$, we gave that $|\alpha| = 1$ if, and only if, $|h(\alpha)| = 1$.
\end{lema}
\begin{proof}

Let $e \in E_1^1$. Suppose that $h(e) = \alpha_1 \alpha_2 \in W_2$, with $\alpha_1, \alpha_2 \notin E_2^0$. Then there exist $f_1, f_2 \in W_1$ such that $h(f_1) = \alpha_1 \text{ and } h(f_2) = \alpha_2$. Notice that by Lemma \ref{lemarangesource} $f_1, f_2 \notin E_1^0$. Therefore, since \[h(f_1f_2) = h(f_1)h(f_2) = \alpha_1 \alpha_2 = h(e),\]
we have that $f_1f_2 = e$ and hence $f_1 \in E_1^0 \text{ or } f_2 \in E_1^0$, what contradicts our assumption. Now, if $h(e) \in E_2^0$ then
\[h(ee) = h(e)h(e) = h(e)\]
and hence $e \in E_1^0$, which is a contradiction. The only possibility left is that $|h(e)|=1$.

Conversely, suppose that $\alpha \in W_1$ is such that $|h(\alpha)| = 1$. Notice that if $\alpha \in E_1^0$ then $h(\alpha) \in E_2^0$ and $|h(\alpha)| = 0$. So suppose that $|\alpha| = n, n \geq 2$. Then $\alpha = \alpha_1 \ldots\alpha_n$, for $\alpha_i \in E_1^1$, and hence
\[h(\alpha) = h(\alpha_1\ldots \alpha_n) = h(\alpha_1)\ldots h(\alpha_n).\]
and $1 = |h(\alpha)| = |h(\alpha_1)\ldots h(\alpha_n)| = n$, a contradiction. Therefore $|\alpha|=1$. 
\end{proof}

\begin{corolario} Let $h: G_1 \to G_2$ be injective in $W_1$ and such that $h(W_1) = W_2$. Then $|\alpha| = |h(\alpha)|$ for all $\alpha \in W_1$.
 \end{corolario}

With the above conditions on the homomorphism $h: G_1 \to G_2$ we can actually show that $h$ is a bijection. We do this below.


\begin{proposition}
Suppose that $h: G_1 \to G_2$ is a homomorphism such that $h|_{W_1}$ is injective and $h(W_1) = W_2$. Then $h$ is bijective.
\end{proposition}\label{lemasobrejetiva}
\begin{proof}
First we prove that $h$ is surjective. 
Let $\gamma \in G_2$. Then we can write $\gamma = \gamma_1 \ldots \gamma_n$, where $\gamma_i \in E_2^0\cup E_2^1 \cup (E_2^1)^*$. Since $h(W_1) = W_2$, $h(W_1^*) = W_2^*$ and $h(E_1^0)=E_2^0$, we have that for each $i$ there exists an unique $e_i \in W_1\cup W_1^*$ such that $\gamma_i = h(e_i)$. By Lemma \ref{lematamanho}, we have that $e_i \in E_1^0 \cup E_1^1 \cup (E_1^1)^*$ and, by Lema \ref{lemacomponivel}, the element $e_1 \ldots e_n \in G_1$. Therefore
\[ h(e_1 \ldots e_n) = h(e_1)\ldots h(e_n) = \gamma_1 \ldots \gamma_n = \gamma.\]

Now suppose that there exists another $f_1\ldots f_k \in G_1$ such that $\gamma = h(f_1 \ldots f_k)= h(f_1)\ldots h(f_k)$. Clearly $k\geq n$, since $\gamma$ is in reduced form. If $k>n$, then $h(f_1)\ldots h(f_k)$ is not in reduced form and therefore there exists $j$ such that $h(f_j)h(f_{j+1})$ is equal to either $s(h(f_j))$ or $r(h(f_{j+1}))$. Suppose, without loss of generality, that $h(f_j)h(f_{j+1})=s(h(f_j))$. Then $h(f_j f_{j+1})=h(s(f_j))$ and, since $h$ is 1-1 in $W_1$, we have that $f_j f_{j+1} = s(f_j)$. Hence $f_1 \ldots f_k$ is not in reduced form, what is a contradiction. Therefore $k=n$ and it follows by Lemma \ref{lematamanho} that $h(f_i)=\gamma_i=h(e_i)$ for all $i$ and hence $f_i = e_i$ for all $i$ (from injectivity of $h$ in $W_1$).

\qedhere
\end{proof}


Our main results in this section concern groupoid graded homomorphisms. We give this definition below and then proceed to prove the two theorems of the section.


\begin{definition}
Let $A = \displaystyle\bigoplus_{g \in G_1} A_g$ and $B = \displaystyle\bigoplus_{h \in G_2}B_h$ be graded algebras by the groupoids $G_1 \text{ e } G_2$, respectively. A graded groupoid homomorphism between $A$ and $B$ is a pair $(\varphi, h)$, where $\varphi: A \to B$ is an algebra homomorphism, $h: G_1 \to G_2 $ is a groupoid homomorphism and $\varphi(A_g)\subseteq B_{h(g)}$, for all $g \in G_1$.
\end{definition}

\begin{theorem}
Let $E_1$ and $E_2$ be graphs and $G_1, G_2$ the associated free path groupoids, respectively. If there exists a groupoid homomorphism $h: G_1 \to G_2$ such that $h|_{W_1}$ is injective and $h(W_1) = W_2$ (where $W_i$ denotes the set of finite paths in $E_i$, $i=1,2$), then the associated Leavitt path algebras $D(X_1) \rtimes G_1$ and $D(X_2) \rtimes G_2$ are groupoid graded isomorphic, via an isomorphism $\varphi : D(X_1)\rtimes G_1 \to D(X_2) \rtimes G_2$ such that $\varphi(1_e\delta_e) = 1_{h(e)}\delta_{h(e)}$ and $\varphi(1_v \delta_v) = 1_{h(v)}\delta_{h(v)}$.
\end{theorem}

\begin{proof}

Consider in $L_K(E_2)$ the family
\[\mathcal{F} = \{ 1_{h(e)}\delta_{h(e)} , 1_{h(e^{-1})}\delta_{h(e^{-1})}, 1_{h(v)}\delta_{h(v)} : v \in E_1^0 , e \in E_1^1\}.\]
We will show that $\mathcal{F}$ satisfy the relations defining the Leavitt path algebra $L_K(E_1)$.

\begin{itemize}
\item[(i)] Since for every vertex $v$ the action $\alpha_v$ is the identity we have that 
\begin{align*}
1_{h(s(e))}\delta_{h(s(e))} \cdot 1_{h(e)}\delta_{h(e)} &= 1_{s(h(e))}\delta_{s(h(e))} \cdot 1_{h(e)}\delta_{h(e)}\\
&= \alpha_{s(h(e))}(\alpha_{s(h(e))}^{-1}(1_{s(h(e))})1_{h(e)})\delta_{s(h(e))h(e)} \\
&= 1_{h(e)}\delta_{h(e)}.
\end{align*}

Also,

\vspace{-20pt}\begin{align*}
1_{h(e)}\delta_{h(e)} \cdot 1_{r(h(e))}\delta_{r(h(e))} &= 1_{h(e)}\delta_{h(e)} \cdot 1_{h(r(e))}\delta_{h(r(e))} \\
&= \alpha_{h(e)}(\alpha_{h(e)}^{-1}(1_{h(e)})1_{r(h(e))})\delta_{h(e)r(h(e))} \\
&= \alpha_{h(e)}(1_{h(e)^{-1}}1_{r(h(e))})\delta_{h(er(e))}\\
&= \alpha_{h(e)}(1_{h(e)^{-1}})\delta_{h(e)}\\
&= 1_{h(e)}\delta_{h(e)}
\end{align*}

\item[(ii)] For the second relation notice that
\vspace{-10pt}\begin{align*}
1_{h(r(e))}\delta_{h(r(e))} \cdot 1_{h(e^{-1})}\delta_{h(e^{-1})} &= 
1_{r(h(e))}\delta_{r(h(e))} \cdot 1_{h(e^{-1})}\delta_{h(e^{-1})}\\
&= \alpha_{r(h(e))}(\alpha_{r(h(e))}^{-1}(1_{r(h(e))})1_{h(e)^{-1}})\delta_{h(e^{-1})} \\
&= 1_{h(e^{-1})}\delta_{h(e^{-1})}
\end{align*}
and 

\vspace{-20pt}\begin{align*}
\hspace{-30pt}1_{h(e^{-1})}\delta_{h(e^{-1})}\cdot 1_{s(h(e))}\delta_{s(h(e))}&=
1_{h(e^{-1})}\delta_{h(e^{-1})}\cdot 1_{s(h(e))}\delta_{s(h(e))} \\
&= \alpha_{h(e^{-1})}(\alpha_{h(e^{-1})}^{-1}(1_{h(e^{-1})})1_{s(h(e))})\delta_{h(e^{-1})s(h(e^{-1}))} \\
&= \alpha_{h(e^{-1})}(1_{h(e^{-1})^{-1}}1_{s(h(e))})\delta_{h(e^{-1}s(e))} \\
&= \alpha_{h(e^{-1})}(1_{h(e)}1_{s(h(e))})\delta_{h(e^{-1})} \\
&= \alpha_{h(e^{-1})}(1_{h(e)})\delta_{h(e^{-1})}\\
&= 1_{h(e^{-1})}\delta_{h(e^{-1})}
\end{align*}


\item[(iii)] Next we check the third relation:

\begin{align*}
1_{h(e^{-1})}\delta_{h(e^{-1})}1_{h(g)}\delta_{h(g)} &= \alpha_{h(e^{-1})}(\alpha_{h(e^{-1})}^{-1}(1_{h(e^{-1})})1_{h(g)})\delta_{h(e^{-1})h(g)} \\
&= \alpha_{h(e^{-1})}(\alpha_{h(e)}(1_{h(e^{-1})})1_{h(g)})\delta_{h(e^{-1})h(g)} \\
&=  \alpha_{h(e^{-1})}(1_{h(e)} 1_{h(g)})\delta_{h(e^{-1})h(g)} \\
&= \delta_{h(e),h(g)} 1_{h(e^{-1})}\delta_{h(r(e))} \\
&= \delta_{e,g} 1_{h(e^{-1})}\delta_{h(r(e))}
\end{align*}

\item[(iv)] Finally we check the Cuntz-Krieger relation. Let $e \in E_1^0$ be such that $0 < \# \{e : s(e) = v \} < \infty$. Notice that 
\begin{align*}\{ e: s(e) = v\} &= \{e: h(s(e)) = h(v)\} = \{e: s(h(e)) = h(v)\}.
\end{align*}
Furthermore,
\begin{align*}
\displaystyle\sum_{s(e)= v} 1_{h(e)}\delta_{h(e)}1_{h(e^{-1})}\delta_{h(e^{-1})} &= \displaystyle\sum_{s(e) = v} 1_{h(e)}\delta_{s(h(e))} = \displaystyle\sum_{s(e) = v} 1_{h(e)}\delta_{h(v)}. \\
\end{align*}
Therefore it is enough to show that $\displaystyle\sum_{e: s(e) = v}1_{h(e)} = 1_{h(v)}$.
But this follows since $$\{ e \in E_2^1: s(e)=h(v) \}=\{h(f):f\in E^1_1 \text{ and }s(f)=v\}$$ and $1_{h(v)}= \displaystyle\sum_{f: s(f) = h(v)}1_{f}$.    
\end{itemize}

We conclude that there exists a homomorphism $\varphi : L_K(E_1) \to L_K(E_2)$ such that $\varphi(1_e\delta_e) = 1_{h(e)}\delta_{h(e)}$, $\varphi(1_{e^{-1}}\delta_{e^{-1}}) = 1_{h(e^{-1})}\delta_{h(e^{-1})}$ and $\varphi(1_v\delta_v) = 1_{h(v)}\delta_{h(v)}$. To show that this $\varphi$ is bijective we construct its inverse in an analogous way to what was done above. Namely, consider in $L_K(E_1)$ the family \[\mathcal{H} = \{ 1_{h^{-1}(e)}\delta_{h^{-1}(e)} , 1_{h^{-1}(e^{-1})}\delta_{h^{-1}(e^{-1})}, 1_{h^{-1}(v)}\delta_{h^{-1}(v)} : v \in E_2^0 , e \in E_2^1\}.\]
Since $h: G_1 \to G_2$ is a bijection such that $h^{-1}(W_2) = W_1$ e $h^{-1}|_{W_2} = W_1$, we obtain, proceeding as above, the  inverse homomorphism of $\varphi$. 

\qedhere

\end{proof}

Under condition (L) for the graph $E_1$, that is, when every closed cycle in $E_1$ has an exit, we can prove the converse of the above result. We do this in two steps.

\begin{theorem}\label{teor2}
Let $E_1$ and $E_2$ be two graphs and suppose that $E_1$ satisfies condition (L). If $(\varphi,h)$ is a graded homomorphism between the associated Leavitt path algebras, where $\varphi: D(X_1) \rtimes G_1 \to D(X_2)\rtimes G_2$ is an isomorphism and  $h: G_1 \to G_2$ is a homomorphism such that $\varphi(\{1_v\delta_v : v \in E_1^0\}) = \{ 1_w\delta_w : w \in E_2^0\}$,
then $h|_{W_1}$ is injective and $h(S_1) = S_2$ (where $S_i$ are as in definition \ref{Sset}).
\end{theorem}

\begin{proof}
We show first that $h(S_1) = S_2$.

Given $g \in S_1$, we have that $1_g\delta_g \neq 0$ in $D(X_1) \rtimes G_1$. Therefore, since $\varphi$ is injective, $0 \neq \varphi(1_g\delta_g) \in D_{h(g)}\delta_{h(g)}$ and hence $h(g) \in S_2$.

For the other inclusion, let $c \in S_2$. Then $1_c\delta_c \neq 0$ and, since $\varphi$ is an isomorphism, there exists $x = \sum \alpha_g \delta_g$ such that $1_c\delta_c = \varphi(x) = \sum \varphi(\alpha_g\delta_g)$. Since $\varphi(\alpha_g\delta_g) \in D_{h(g)}\delta_{h(g)}$, for all $g$, we have that 
there exists $g$ such that $h(g) = c$ and $\varphi(\alpha_g\delta_g) \neq 0$ (notice that $g\in S_1$, since $\alpha_g\delta_g \neq 0$). 

Next we prove that $h|_{W_1}$ is injective. We divide this in three steps. Recall that, by Lemma \ref{lemarangesource}, $h(E_1^0) \subseteq E_2^0$. 

\textbf{Step 1: } $h$ is injective in the vertices, that is, $h|_{E_0^1}$ is injective.

Let $v, w \in E_1^0$. Notice that
\[\varphi(1_v\delta_v 1_w\delta_w) = \varphi(1_v\delta_v)\varphi(1_w\delta_w) = 1_{h(v)}\delta_{h(v)}\cdot 1_{h(w)}\delta_{h(w)} = \delta_{h(v), h(w)}1_{h(v)}\delta_{h(v)}.\]
On the other hand, $\varphi(1_v\delta_v 1_w\delta_w) = \varphi(\delta_{v,w} 1_v\delta_v) = \delta_{v,w}1_{h(v)}\delta_{h(v)}$. Therefore $h(v) = h(w)$ if, and only if, $v=w$.

\textbf{Step 2: } Let $\alpha \in S_1$. If $h(\alpha) \in E_2^0$ then $\alpha \in E_1^0$.

Suppose that $\alpha \in S_1$ is such that $h(\alpha) \in E_2^0$. Then we have that 
\[h(s(\alpha)) = s(h(\alpha)) = r(h(\alpha)) = h(r(\alpha)).\]
Therefore, since $h$ is injective on the vertices (Step 1), we have that $s(\alpha) = r(\alpha)$ and $\alpha$ must be a closed cycle. 

Next notice that, since the partial action on the ideals associated to vertices is the identity, we have that
\[\varphi(1_{\alpha}\delta_{\alpha}1_{\alpha^{-1}}\delta_{\alpha^{-1}}) =\varphi(1_{\alpha}\delta_{\alpha})\varphi(1_{\alpha{-1}}\delta_{\alpha^{-1}}) = a\delta_{h(\alpha)}b\delta_{h(\alpha^{-1})} = ab\delta_{h(\alpha)},\]
where $a, b \in D_{h(\alpha)}= D_{h(\alpha^{-1})}$ are such that $\varphi(1_{\alpha}\delta_{\alpha}) = a\delta_{h(\alpha)}$ and $\varphi(1_{\alpha^{-1}}\delta_{\alpha^{-1}}) = b\delta_{h(\alpha^{-1})} = b\delta_{h(\alpha)}$.
Analogously, $\varphi(1_{\alpha^{-1}}\delta_{\alpha^{-1}} 1_{\alpha}\delta_{\alpha}) = ba\delta_{h(\alpha)} = ab\delta_{h(\alpha)}$. Since $\varphi$ is an isomorphism we have that $1_{\alpha}\delta_{\alpha} 1_{\alpha^{-1}}\delta_{\alpha^{-1}} = 1_{\alpha^{-1}}\delta_{\alpha^{-1}}1_{\alpha}\delta_{\alpha}$, that is,
\[1_{\alpha}\delta_{s(\alpha)} = 1_{\alpha^{-1}}\delta_{r(\alpha)}.\]
Hence, since $s(\alpha)=r(\alpha)$, the above equality implies that $1_{\alpha} = 1_{\alpha^{-1}} = 1_{r(\alpha)}$, what can only happen if $\alpha\in E_1^0$ (as we show below).

Suppose that $\alpha \in W_1$ or $\alpha \in W_1^*$. Without loss of generality we can assume that $\alpha \in W_1$. Since $E_1$ satisfies condition $(L)$, the cycle $\alpha$ has an exit $e \in E_1^1$. So we can choose a subpath $\alpha'$ of $\alpha$ such that $\alpha = \alpha' \alpha''$ and $\alpha'e$ is composable. Then $1_{\alpha}(\alpha'e) = 0$ and $1_{r(\alpha)}(\alpha'e) = 1$ and hence $\alpha \notin W_1\cup W_1^*$.

Finally, suppose that $\alpha = \beta\gamma^{-1}$, where $\beta, \gamma \in W_1 \setminus E_1^0$ and $r(\beta) = r(\gamma)$, in reduced form. Then $1_{\beta} = 1_{\alpha} = 1_{\alpha^{-1}} = 1_{\gamma},$ what is impossible. 

We conclude that $\alpha \in E_1^0$ as desired.

\textbf{Step 3: } $h$ is injective in $W_1$.

Let $\alpha$ and $\beta \in W_1$ be such that $h(\alpha) = h(\beta)$. Then
\[ h(r(\alpha)) = r(h(\alpha)) = r(h(\beta)) = h(r(\beta)),\]
and by Step 1 we have that $r(\alpha) = r(\beta)$. Therefore $\alpha\beta^{-1}$  is composable and \[h(\alpha\beta^{-1}) = h(\alpha)h(\beta)^{-1} = h(\alpha)h(\alpha)^{-1} = s(h(\alpha)) \in E_2^0.\]
By Step 2 we have that $\alpha\beta^{-1} \in E_1^0$ and hence $\alpha = \beta$ as desired.
\end{proof}

\begin{corolario}\label{ultimaprop}
If in Theorem \ref{teor2} we further assume that $\varphi(\{1_e\delta_e : e \in E_1^1\}) = \{1_f \delta_f : f \in E_2^1\}$, then $h(W_1) = W_2$.
\end{corolario}

\begin{proof}
Notice that with the additional hypothesis, since $(\varphi,h)$ is a graded homomorphism, then $\varphi(1_e\delta_e) = 1_{h(e)}\delta_{h(e)}$, for all $e \in E_1^1$.

Let $e_1 \ldots e_n \in W_1$. Then $h(e_1 \ldots e_n) = h(e_1) \ldots h(e_n)$ and since $h(e_i) \in E_2^1$, for all $i \in \{1, \ldots , n\}$, we have that $h(e_1 \ldots e_n) \in W_2$ and hence $h(W_1) \subseteq W_2$.

For the other inclusion, let $f \in E_2^1$. Since $\varphi(\{1_e\delta_e : e \in E_1^1\}) = \{1_f \delta_f : f \in E_2^1\}$, there exists $e \in E_1^1$ such that $1_f\delta_f = \varphi(1_e\delta_e)= 1_{h(e)}\delta_{h(e)}$ and hence $h(e) = f$. So, given  $f_1 \ldots f_n \in W_2$ there exist $e_1, \ldots, e_n \in W_1$ such that $h(e_i) = f_i$ and hence, by Lemma \ref{lemacomponivel}
\[ f_1 \ldots f_n = h(e_1)\ldots h(e_n) = h(e_1 \ldots e_n).\] 
Therefore $W_2 \subseteq h(W_1)$ as desired.

\qedhere

\end{proof}

We end our work with an example that shows that the additional hypothesis included in Corollary \ref{ultimaprop} is essential. 

\begin{example}
Let $E_1 = \{E_1^0, E_1^1, r_1, s_1\}$, where $E_1^0 = \{ v_1, v_2, v_3\}$, $E_1^1 = \{e_1, e_2\}$, $r_1(e_1) = v_2 = s_1(e_2)$, $s_1(e_1) = v_1$ and $r_1(e_2) = v_3$. We picture this graph below.
$$ \xymatrix{ v_1 \ar[r]^{e_1}  & v_2 \ar[r]^{e_2} & v_3 } $$
Let $E_2 = \{E_2^0, E_2^1, r_2, s_2\}$, where $E_2^0 = \{ w_1, w_2, w_3\}$, $E_2^1 = \{f_1, f_2\}$, $r_2(f_1) = w_3 = r_2(f_2)$, $s_2(f_1) = w_1$ and $s_2(f_2) = w_2$. We picture this graph below
$$ 
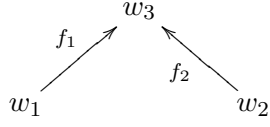
\xymatrix{ & w_3  &\\ w_1 \ar[ur]^{f_1} &  & w_2 \ar[ul]^{f_2} } $$

Notice that $G_1 = \{v_1, v_2, v_3, e_1, e_2, e_1^{-1}, e_2^{-1}, e_1e_2, e_2^{-1}e_1^{-1}\}$.

Define $h: G_1 \to G_2$ so that $h(e_1) = f_1^{-1}$ and $h(e_2) = f_1f_2^{-1}$, 
Notice that
\begin{align*}
h(v_1) &= h(e_1e_1^{-1}) = h(e_1)h(e_1)^{-1} = f_1^{-1}f_1 = r(f_1) = w_3 \\
h(v_2) &= h(e_1^{-1}e_1) = h(e_1)^{-1}h(e_1)= f_1f_1^{-1} = s(f_1) = w_1 \\
h(v_3) &= h(e_2^{-1}e_2) = h(e_2)^{-1}h(e_2) = (f_1f_2^{-1})^{-1}(f_1f_2^{-1}) \\ &= f_2 r(f_1) f_2^{-1} = s(f_2) = w_2.
\end{align*}

Using the universal property of Leavitt path algebras one can construct a graded isomorphism $(\varphi,h)$, where $\varphi: L_K(E_1) \to L_K(E_2)$ is such that $\varphi(1_{e_1}\delta_{e_1}) = 1_{f_1^{-1}}\delta_{f_1^{-1}}$, $\varphi({1_{e_2}\delta_{e_2}}) = 1_{f_1f_2^{-1}}\delta_{f_1f_2^{-1}}$, $\varphi(1_{v_1}\delta_{v_1}) = 1_{w_3}\delta_{w_3}$, $\varphi(1_{v_2}\delta_{v_2}) = 1_{w_1}\delta_{w_1}$, $\varphi(1_{v_3}\delta_{v_3}) = 1_{w_2}\delta_{w_2}$.

Notice that $\varphi(\{1_{v_1}\delta_{v_1}, 1_{v_2}\delta_{v_2}, 1_{v_3}\delta_{v_3}\}) = \{1_{w_1}\delta_{w_1}, 1_{w_2}\delta_{w_2}, 1_{w_3}\delta_{w_3}\}$, but $\varphi(\{1_{e_1}\delta_{e_1}, 1_{e_2}\delta_{e_2}\}) \neq \{1_{f_1}\delta_{f_1}, 1_{f_2}\delta_{f_2}\}$, that is, the additional hypothesis of Corollary \ref{ultimaprop} is not satisfied. Furthermore, it is clear that $h(W_1) \neq W_2$. 

\qed

\end{example}

\vspace{1.5pc}

Daniel Gonçalves, Departamento de Matemática, Universidade Federal de Santa Catarina, Florianópolis, 88040-900, Brasil

Email: daemig@gmail.com

\vspace{0.5pc}
Gabriela Yoneda, Departamento de Matemática, Universidade Federal de Santa Catarina, Florianópolis, 88040-900, Brasil

Email: yonedagabriela@gmail.com
\vspace{0.5pc}

\end{document}